\documentclass[number,sort&compress]{elsarticle}

\usepackage{amssymb}
\usepackage{amsmath}
\usepackage{amsthm}
\usepackage{cleveref}
\crefname{equation}{}{}
\crefname{section}{Section}{Sections}
\Crefname{section}{Section}{Sections}
\crefname{subsection}{Subsection}{Subsections}
\Crefname{subsection}{Subsection}{Subsections}
\Crefname{equation}{Equation}{Equations}
\crefname{theorem}{Theorem}{Theorems}
\Crefname{theorem}{Theorem}{Theorems}
\crefname{lemma}{Lemma}{Lemmas}
\Crefname{lemma}{Lemma}{Lemmas}
\crefname{corollary}{Corollary}{Corollaries}
\Crefname{corollary}{Corollary}{Corollaries}
\crefname{conjecture}{Conjecture}{Conjectures}
\Crefname{conjecture}{Conjecture}{Conjectures}
\crefname{observation}{Observation}{Observations}
\Crefname{observation}{Observation}{Observations}

\newcommand{\Z}{\mathbb{Z}}
\newcommand{\F}{\mathbb{F}}

\newcommand{\s}{\mathsf{s}}

\newtheorem{theorem}{Theorem}[section]
\newtheorem{lemma}[theorem]{Lemma}
\newtheorem{corollary}[theorem]{Corollary}
\newtheorem{conjecture}[theorem]{Conjecture}
\theoremstyle{definition}
\newtheorem{observation}[theorem]{Observation}
\numberwithin{equation}{section}

\title{
On generalized Erd\H{o}s--Ginzburg--Ziv constants 
for $\Z_2^{\lowercase{d}}$
}

\author{Alexander Sidorenko}
\address{sidorenko.ny@gmail.com}
\ead{sidorenko.ny@gmail.com}

\date{\today}

\begin{document}

\begin{abstract}
Let $G$ be a finite abelian group, 
and $r$ be a multiple of its exponent. 
The generalized Erd\H{o}s--Ginzburg--Ziv constant $\s_r(G)$
is the smallest integer $s$
such that every sequence of length $s$ over $G$ 
has a zero-sum subsequence of length $r$. 
We find exact values of $\s_{2m}(\mathbb{Z}_2^d)$ for $d \leq 2m+1$. 
Connections to 
linear binary codes of maximal length 
and codes without a forbidden weight 
are discussed.
\end{abstract}


\maketitle

\section{Introduction}

Let $G$ be a finite abelian group written additively. 
We denote by $\exp(G)$ the {\it exponent} of $G$
that is the least common multiple of the orders of its elements. 
Let $r$ be a multiple of $\exp(G)$. 
The generalized Erd\H{o}s--Ginzburg--Ziv constant $\s_r(G)$
is the smallest integer $s$
such that every sequence of length $s$ over $G$ 
has a zero-sum subsequence of length $r$. 
If $r = \exp(G)$, then $\s(G)=\s_{\exp(G)}(G)$ 
is the classical Erd\H{o}s--Ginzburg--Ziv constant.

In the case when $k$ is a power of a prime, 
Gao \cite{Gao:2003} proved
$\s_{km}(\Z_k^d) = km + (k-1)d$ for $m \geq k^{d-1}$ 
and conjectured that the same equality holds when $km > (k-1)d$. 

In this paper, we consider the case $G = \Z_2^d$. 
We show that the problem of determining $\s_{2m}(\Z_2^d)$ 
is essentially equivalent to finding the lowest redundancy 
of a linear binary code of given length 
which does not contain words of Hamming weight $2m$. 
When $m=2$, this problem is also equivalent 
to finding the maximal length 
of a linear binary code of redundancy $d$ and distance $5$ or higher.

We prove that 
$\s_{2m}(\Z_2^d) = 2m+d$ for $d < 2m$, 
validating the Gao's conjecture for $k=2$. 
We also prove $\s_{2m}(\Z_2^{2m}) = 4m+1$, 
$\s_{2m}(\Z_2^{2m+1}) = 4m+2$ for even $m$, 
and $\s_{2m}(\Z_2^{2m+1}) = 4m+5$ for odd $m$. 
Our results provide counterexamples 
to Conjectures~4.4 and~4.6 from \cite{Gao:2014}.

This paper is organized as follows. 
We discuss 
maximal length linear binary codes 
in \cref{sec:codes},
linear codes without a forbidden weight 
in \cref{sec:forbidden}, 
and 
generalized Erd\H{o}s--Ginzburg--Ziv constants 
in \cref{sec:EGZG}. 
We present our results for $\s_{2m}(\Z_2^d)$ in \cref{sec:summary}. 
\Cref{sec:proofs} contains the proofs.

\section{Linear binary codes of maximal length}\label{sec:codes}

In this section, we will provide basic definitions 
and some results from coding theory 
(for details, see~\cite{Tomlinson:2017}).

Let $\F_2$ be the binary field 
and $\F_2^n$ be the $n$-dimensional vector space over $\F_2$. 
The Hamming weight of vector $x\in\F_2^n$ is the number of its entries equal to $1$. 
The dot product of vectors 
$x=(x_1,x_2,\ldots,x_n)$ and $y=(y_1,y_2,\ldots,y_n)$ 
is defined as $x \cdot y = x_1 y_1 + x_2 y_2 + \ldots + x_n y_n$. 
A {\it linear binary code} of length $n$ is a subspace in $\F_2^n$. 
Its elements are called {\it words}. 
The {\it distance} of a linear code is 
the smallest Hamming weight of its nonzero word. 
A trivial code of dimension $0$ has distance $\infty$. 
A linear binary code $C$ is called an $(n,k,d)$ code
when it has length $n$, dimension $k$ and distance $d$. 
The {\it dual code} $C^\perp$ 
is the subspace of vectors in $\F_2^n$ orthogonal to $C$. 
The {\it redundancy} of $C$ is the dimension of its dual code, $r=n-k$, 
which may be interpreted as the number of parity check bits.
If $y^{(1)},y^{(2)},\ldots,y^{(r)}$ form a basis of $C^\perp$, 
then $C$ consists of vectors $x$ such that 
$x \cdot y^{(i)} = 0$ for every $i=1,2,\ldots,r$. 
If $y^{(i)} = (y_{i1},y_{i2},\ldots,y_{in})$, 
the $(r \times n)$-matrix $[y_{ij}]$ 
is called a {\it parity-check matrix} of $C$. 
In fact, any binary $r \times n$ matrix of rank $r$ 
is a parity-check matrix of some linear code of length $n$ and redundancy $r$.

An $(n,k,2t+1)$ code is capable of correcting up to $t$ errors 
in a word of length $n$ that carries $k$ bits of information. 
It is natural to seek codes of maximal possible length 
with prescribed error-correction capabilities. 
We denote by $N(r,d)$ the largest length of a linear code 
with redundancy $r$ and distance $d$ or higher, 
that is the largest $n$ such that an $(n,n-r,\geq d)$ code exists. 

It follows from the well known Hamming bound (see \cite{Tomlinson:2017}) that
\begin{equation}\label{eq:Hamming}
  \sum_{i=0}^t \binom{N(r,2t+1)}{i} \; \leq \; 2^r \; .
\end{equation}
The primitive binary BCH code (see \cite{Bose:1960,Hocquenghem:1959}) 
is a $(2^m-1,2^m-1-mt,2t+1)$ code. 
It gives the lower bound
\begin{equation}\label{eq:BCH}
  N(mt,2t+1) \; \geq \; 2^m-1 \; .
\end{equation}
It is easy to see that $N(m,3) = 2^m - 1$.
When $t \geq 2$, the bound \eqref{eq:BCH} is not sharp: 
some codes of slightly larger length are known. 
Goppa~\cite{Goppa:1971} constructed $(2^m,2^m-mt,2t+1)$ codes.
Chen~\cite{Chen:1991} found $(2^m+1,\:2^m+1-2m,\:5)$ codes for even $m$.
Sloane, Reddy, and Chen~\cite{Chen:1991,Sloane:1972} obtained 
$(2^m+2^{\lceil m/2 \rceil}-1,$ $2^m+2^{\lceil m/2 \rceil}-1-(2m+1),\:5)$ 
codes. 
Hence,
\[
  N(4s,  5) \; \geq \; 2^{2s} + 1 \; ,
    \;\;\;\;\;\;
  N(4s+2,5) \; \geq \; 2^{2s+1}   \; ,
\]
\[
  N(4s+1,5) \; \geq \; 2^{2s}   + 2^{s}   - 1  \; ,
    \;\;\;\;\;\;
  N(4s+3,5) \; \geq \; 2^{2s+1} + 2^{s+1} - 1  \; .
\]

The values of $N(r,d)$ for small $r$ and $d$ 
can be derived from tables in \cite{Grassl:codetables}. 
We list these values for $4 \leq r \leq 14,\; d=5$ :

\vspace{3mm}
\noindent
\begin{tabular}{cccccccccccc}
  $r$      & 4 & 5 & 6 &  7 &  8 &  9 & 10 & 11 & 12 & 13 & 14
    \\
  $N(r,5)$ & 5 & 6 & 8 & 11 & 17 & 23 & 33 & 47--57 & 65--88 & 81--124 & 128--179
\end{tabular}
\vspace{3mm}

It follows from \eqref{eq:Hamming} that $N(r,5) \leq 2^{(r+1)/2}$. 
When $r$ is large, the best known lower and upper bounds for $N(r,5)$ 
differ by a factor of $\sqrt{2}$ if $r$ is even,
and by a factor of $2$ if $r$ is odd. 
For $a>1$, not a single $(2^m+a,\: 2^m+a-2m,\: 5)$ code is known.

For future use we need

\begin{theorem}[MacWilliams identities \cite{MacWilliams:1963}]
Let $C$ be a $k$-dimensional linear binary code of length $n$. 
Let $A_j$ denote the number of words of Hamming weight $j$ in $C$, 
and $B_j$ denote the number of words of Hamming weight $j$ 
in the dual code $C^\perp$. 
Then for every $\lambda = 0,1,\ldots,n$,
\begin{equation}\label{eq:MWI}
  2^{n-k} \sum_{j=0}^\lambda \binom{n-j}{\lambda-j} A_j
    \; = \;
  2^\lambda \sum_{j=0}^{n-\lambda} \binom{n-j}{\lambda} B_j
    \; .
\end{equation}
\end{theorem}

\section{Codes without a forbidden weight}\label{sec:forbidden}

Let $R_{2m}(n)$ be the smallest redundancy of a linear code of length $n$ 
which has no words of Hamming weight $2m$. 
The problem of determining $R_{2m}(n)$ was studied in 
\cite{Bassalygo:2000,Enomoto:1987} 
in notation $l(n,\overline{2m}) = n - R_{2m}(n)$. 
It follows from Theorem~1.1 of~\cite{Enomoto:1987} that 
\begin{equation}\label{eq:Enomoto1}
  R_{2m}(n) \; = \;  n - 2m + 1
  \;\;\;\;\; \mbox{for} \;\; 2m-1 \leq n \leq 4m-1 \; ,
\end{equation}
\begin{equation}\label{eq:Enomoto2}
  R_{2m}(4m) \; = \;  2m \; .
\end{equation}
We will solve some new cases with $n > 4m$ in Corollary~\ref{th:Enomoto_my_1}. 

It follows from Theorem~6 of~\cite{Bassalygo:2000} that 
\begin{equation}\label{eq:Bassalygo}
   R_{2m}(n) / m \; = \; \log_2 n \; + \; O(1) \; ,
\end{equation}
when $m$ is fixed and $n\rightarrow\infty$. 

The proof of \eqref{eq:Enomoto1} and \eqref{eq:Enomoto2} in~\cite{Enomoto:1987}
uses the notion of binormal form of a binary matrix. 
Following~\cite{Enomoto:1987}, 
we say that a $k \times n$ binary matrix $M=[a_{ij}]$ 
is in {\it binormal form} if $n \geq 2k$, 
$\;a_{i,2j-1} = a_{i,2j}$ for $i \neq j$, 
and $a_{i,2i-1} \neq a_{i,2i}\;$ 
($i,j = 1,2,\ldots,k$).

\begin{lemma}[Proposition~2.1 \cite{Enomoto:1987}]\label{th:Enomoto1}
If $k \times n$ binary matrix $M$ is in binormal form, 
then for any $k$-dimensional binary vector $x$, 
there is a unique choice of $k$ indices $j_i \in \{ 2i-1,2i \}$
$\; (i=1,2,\ldots,k)$ 
such that the sum of columns $j_1,j_2,\ldots,j_k$ of $M$ is equal to $x$. 
In particular, 
one can pick up $k$ columns in $M$ whose sum is 
the $k$-dimensional zero vector.
\end{lemma}

\begin{lemma}[Lemma~2.2 \cite{Enomoto:1987}]\label{th:Enomoto2}
Let $n$ be odd, $2k < n$, 
and $M$ be a $k \times n$ binary matrix of rank $k$. 
If the sum of entries in each row is $0$, 
then M can be brought to binormal form by such operations as 
permutations of the columns and additions of one row to another.
\end{lemma}

Lemmas \ref{th:Enomoto1} and \ref{th:Enomoto2} yield 

\begin{corollary}\label{th:Enomoto}
Let $n$ be odd and $2k<n$. 
Let $M$ be a $k \times n$ binary matrix of rank $k$ 
where the sum of entries in each row is $0$. 
One can pick up $k$ columns in $M$ whose sum is the $k$-dimensional zero vector.
\end{corollary}

\section{Generalized Erd\H{o}s--Ginzburg--Ziv constant}
\label{sec:EGZG}

Let $G$ be a finite abelian group written additively. 
The classical Erd\H{o}s--Ginzburg--Ziv constant $\s(G)$ is the smallest integer $s$ 
such that every sequence of length $s$ over $G$ 
has a zero-sum subsequence of length $\exp(G)$ 
(see 
\cite{Edel:2007,Ellenberg:2017,Gao:2006,Gao:2003,Harborth:1973,Kemnitz:1983,Reiher:2007}). 
In 1961, Erd\H{o}s, Ginzburg, and Ziv \cite{Erdos:1961} proved 
$\s(\Z_k)=2k-1$. 
Kemnitz' conjecture, $\s(\Z_k^2)=4k-3$ (see~\cite{Kemnitz:1983}), 
was open for more than twenty years 
and finally was proved by Reiher~\cite{Reiher:2007} in 2007.

The following generalization of the classical Erd\H{o}s--Ginzburg--Ziv constant
was introduced by Gao~\cite{Gao:2003}. 
If $r$ is a multiple of $\exp(G)$, 
then $\s_r(G)$ denotes the smallest integer $s$
such that every sequence of length $s$ over $G$ 
has a zero-sum subsequence of length $r$. 
(Notice that if $r$ is not a multiple of $\exp(G)$, 
then there is an element $x \in G$ whose order is not a divisor of $r$,
and the infinite sequence 
$x,x,x,\ldots$ contains no zero-sum subsequence of length $r$.) 
Obviously, $\s_{\exp(G)}(G)=\s(G)$. 
Constants $\s_r(G)$ were studied in 
\cite{Bitz:2017,Gao:2014,Gao:2006b,Gao:2003,Han:2018,Han:2019,He:2016,Kubertin:2005}.

A sequence that consists of $(km-1)$ copies of the zero vector 
and $(k-1)$ copies of each of the basis vectors demonstrates that
\begin{equation}\label{eq:s_lower}
  \s_{km}(\Z_k^d) \; \geq \; km + (k-1)d \; .
\end{equation}
If $km \leq (k-1)d$, we can add $(1,1,\dots,1)$ to this sequence. 
Hence,
\begin{equation}\label{eq:s_lower2}
  \s_{km}(\Z_k^d) \; \geq \; km + (k-1)d + 1
    \;\;\;\;\;\mbox{when}\;\;\; km \leq (k-1)d \; .
\end{equation}
It is easy to see that 
\begin{equation}\label{eq:recursive}
  \s_{km}(\Z_k^d) + (k-1) \; \leq \; \s_{km}(\Z_k^{d+1}) \; .
\end{equation}
Indeed, consider a sequence $S$ over $\Z_k^d$ 
that does not have zero-sum subsequences of size $km$. 
Attach $0$ to each vector in $S$ as the $(d+1)$th entry
and add to the sequence $(k-1)$ copies of a vector 
whose $(d+1)$th entry is $1$. 
The resulting sequence over $\Z_k^{d+1}$ 
will not contain a zero-sum subsequence of length $km$, either. 

In the case when $k$ is a power of a prime, 
Gao~\cite{Gao:2003} proved the equality in \eqref{eq:s_lower} 
for $m \geq k^{d-1}$ 
and conjectured 
\begin{equation}\label{eq:Gao}
  \s_{km}(\Z_k^d) \; = \; km + (k-1)d 
    \;\;\;\;\;\mbox{for}\;\;\; km > (k-1)d \; .
\end{equation}

The connection between generalized Erd\H{o}s--Ginzburg--Ziv constants of $\Z_2^d$  
and linear binary codes is evident from the following observation. 
Let $S$ be a sequence of length $n$ over $\Z_2^d$. 
Write its $n$ vectors column-wise to get a $d \times n$ binary matrix $M$. 
Obviously, $S$ has a zero-sum subsequence of length $r$ if and only if 
$M$ has $r$ columns that sum up to a zero vector. 
Let $C$ be the subspace in $\Z_2^n$ generated by the rows of $M$.
If $M$ has $r$ columns that sum up to a zero vector, 
then the same will be true for any basis of $C$ written row-wise. 
The $n$-dimensional vector, 
whose non-zero entries are positioned in these $r$ columns, 
will be orthogonal to any word of $C$. 
It means that the dual code $C^\perp$ has a word of weight $r$. 
The same arguments work in the opposite way, too. 
If a linear binary code has a word of weight $r$ 
then any its parity check matrix has $r$ columns 
that sum up to a zero vector. 

When $k>2$ is a power of a prime, a similar connection exists between 
the generalized Erd\H{o}s--Ginzburg--Ziv constants of $\Z_k^d$ 
and linear $k$-ary codes (which are subspaces of vector spaces over field $\F_k$), 
but unfortunately, it works only one way. 
If a sequence over $\F_k^d$ has a zero-sum subsequence of length $r$, 
then being written column-wise, it serves as a parity check matrix of a $k$-ary code 
which has a word whose $r$ entries are equal to $1$ and the rest are equal to $0$. 
However, the fact that a $k$-ary code has a word with $r$ nonzero entries 
does not guarantee that its parity-check matrix has $r$ columns 
that sum up to a zero vector.

\section{Summary of results}\label{sec:summary}

In this section, we consider the case $G=\Z_2^d$. 
We will show (see Theorem~\ref{th:s_m_small}) 
that Gao's conjecture \eqref{eq:Gao} holds for $k=2$.
We will also show that 
constants $N(d,5)$ from \Cref{sec:codes} and $\s_4(\Z_2^d)$ 
are equivalent (namely, $\s_4(\Z_2^d) = N(d,5)+4$) 
while constants $R_{2m}(n)$ from \cref{sec:forbidden} 
and $\s_{2m}(\Z_2^d)$ are closely related. 

To simplify notation, we will write $\s_{2m}(d)$ instead of $\s_{2m}(\Z_2^d)$. 
Let $W$ be a set of positive integers which contains at least one even number. 
We denote by $\beta_W(d)$ the largest size of a set in $\Z_2^d$ 
which has no zero-sum subsets of size $w \in W$. 
We will use the following shortcuts:
\[
  \beta_{2m}(d) = \beta_{\{2m\}}(d), \;\;\;\;
  \beta_{2[k,m]}(d) = \beta_{\{2k,2k+2,\ldots,2m\}}(d), 
\]
\[
  \beta_{[1,2m]}(d) = \beta_{\{1,2,\ldots,2m\}}(d).
\]
As it should be expected, $\s_{2m}(d)$ and $\beta_{2m}(d)$ are close:
\begin{equation}\label{eq:beta_s_beta}
  \beta_{2m}(d) + 1
    \; \leq \;
  \s_{2m}(d)
    \; \leq \;
  \beta_{2m}(d) + 2m-1
  \; .
\end{equation}
The lower bound in \eqref{eq:beta_s_beta} is trivial.
The upper bound was proved in \cite[Theorem~4.1]{Sidorenko:2017}.
The set of $d$ basis vectors in $\Z_2^d$ 
with addition of vector $(1,1,\ldots,1)$ 
demonstrates that 
\begin{equation}\label{eq:even_beta_1}
  \beta_{[1,2m]}(d) \; \geq \; d+1 \;\;\;\;\mbox{for}\;\; d \geq 2m \; .
\end{equation}
If $d \geq 3m$, there exist vectors $x,y\in\Z_2^d$ such that 
each of the three vectors $x,y,x+y$ has Hamming weight $2m$. 
Then $x,y$ and the $d$ basis vectors demonstrate that 
\begin{equation}\label{eq:even_beta_2}
  \beta_{[1,2m]}(d) \; \geq \; d+2 \;\;\;\;\mbox{for}\;\; d \geq 3m \; .
\end{equation}

\begin{theorem}\label{th:beta_123_2m}
$\;
  \beta_{[1,2m]}(d) \; = \; N(d,2m+1) \; .
$
\end{theorem}

\begin{theorem}\label{th:beta_24_2m}
$\;
  \beta_{2[1,m]}(d) \; = \; \beta_{[1,2m]}(d) + 1 \; .
$
\end{theorem}

It is easy to see that $\beta_{2m}(d) = 2^d\;$ if $\;2m > 2^d$. 

\begin{theorem}\label{th:beta_d_small}
$
  \beta_{2m}(d) \; = \; 
  \begin{cases} 
   2m + 2, & \mbox{if }\; 2^{d-1} \leq 2m < 2^d,\; m \neq 2^{d-1} - 2, \\
   2m    , & \mbox{if }\; m = 2^{d-1} - 2.
  \end{cases}
$
\end{theorem}

\begin{theorem}\label{th:beta_d-1}
Let $m$ be odd. 
Then $\beta_{2m}(d) \geq 2\beta_{2[1,m]}(d-1)$. 
If $b$ is even and $m < b \leq \beta_{2[1,m]}(d-1)$, 
then $\beta_{2b-2m}(d) \geq 2b$.
\end{theorem}

\begin{theorem}\label{th:s_max_beta}
${\displaystyle \; 
  \s_{2m}(d) \; = \;
  1 + \max_{1 \leq j \leq m} \left\{
    \beta_{2[j,m]}(d) + (2m-2j)
  \right\}
}$.
\end{theorem}

\begin{theorem}\label{th:beta_s_4}
$\;
  \s_4(d)
    \; = \;
  \beta_4(d) + 3
    \; = \;
  N(d,5) + 4
    \; .
$
\end{theorem}

\begin{theorem}\label{th:beta_s_6}
$\;
  \s_6(d)
    \; = \;
  \beta_6(d) + 1
    \;
$
for $d \geq 3$.
\end{theorem}

$R_{2m}(n)$ and $\s_{2m}(d)$ are related, because 
$R_{2m}(n)$ is, in fact, the smallest number of rows 
in a binary matrix with $n$ columns 
that does not contain $2m$ columns summing up to a zero vector,
while $\s_{2m}(d)$ is the largest number of columns 
in a matrix with $d$ rows that has the same property.

\begin{theorem}\label{th:R_s}
$\;
  R_{2m}(n) = d 
    \; \iff \;
  \s_{2m}(d-1) \; \leq \; n < \s_{2m}(d) \; .
$
\end{theorem}

\begin{theorem}\label{th:s_m_small}
$\;
  \s_{2m}(d) \; = \; 2m + d 
    \;\;\; \mbox{for} \;\; d < 2m \; .
$
\end{theorem}

\begin{theorem}\label{th:s_2m_2m}
$\;
  \s_{2m}(2m) \; = \; 4m+1 \; . 
$
\end{theorem}

Theorems \ref{th:beta_24_2m} and \ref{th:beta_d-1} 
together with \eqref{eq:even_beta_1} and \eqref{eq:even_beta_2} yield

\begin{corollary}\label{th:s_lower3}
If $m$ is odd, then 
$\beta_{2m}(d) \geq 2d+2\;$ for $d \geq 2m+1$, 
and 
$\beta_{2m}(d) \geq 2d+4\;$ for $d \geq 3m+1$. 
\end{corollary}

\begin{theorem}\label{th:s_2m_2m_1_odd}
$\;
  \s_{2m}(2m+1) \; = \; 4m+5 \:  
$
for odd $m$.
\end{theorem}

\begin{theorem}\label{th:beta_odd}
If $m \geq 3$ is odd and $2m-3 \leq d \leq 2m+1$, 
then $\beta_{2m}(d) = \s_{2m}(d) - 1$.
\end{theorem}

In line with Corollary \ref{th:s_lower3} 
and Theorem \ref{th:s_2m_2m_1_odd}, we propose

\begin{conjecture}\label{conj:odd}
If $m$ is odd, 
$\;\s_{2m}(d) = 2d+3\:$ for $\:2m+1 \leq d \leq 3m$.
\end{conjecture}

When $d > 2m$, the cases of even and odd $m$ differ significantly.
For even $m$ and $2m < d < 3m$, 
the best lower bound we know  
follows from \eqref{eq:recursive} and Theorem~\ref{th:s_2m_2m}:
$\:
  \s_{2m}(d) \geq \s_{2m}(2m) + (d-2m) = d + 2m + 1 .
$

\begin{theorem}\label{th:s_2m_2m_1_even}
$\;
  s_{2m}(2m+1) \; = \; 4m+2 \;  
$
for even $m$.
\end{theorem}

\begin{conjecture}\label{conj:even}
If $m$ is even, 
$\:\s_{2m}(d) = d + 2m + 1\:$ for $2m+1 \leq d \leq 3m-1$.
\end{conjecture}

To prove Conjecture \ref{conj:even}, it would be sufficient to show that 
$\s_{4k}(6k-1) \leq 10k$. 
This is equivalent to the statement that 
every $(4k+1)$-dimensional code of length $10k$ has a word of weight $4k$.

Computer search confirmed that 
Conjectures~\ref{conj:odd} and~\ref{conj:even} 
hold for $m=3$ and $m=4$, respectively.

Conjecture 4.4 of \cite{Gao:2014} in the case $G=\Z_2^d$ claims 
$\s_2(d) - 2 > \s_4(d) - 4 > \cdots > \s_{2m}(d) - 2m$, 
where $m = \lfloor d/2 \rfloor$. 
Theorems~\ref{th:beta_s_4} and~\ref{th:s_2m_2m_1_odd} 
provide a counterexample: 
$\s_4(7) - 4 = \s_6(7) - 6 = 11$. 
Computer search showed that $\s_8(11)=20$, 
and by Theorem~\ref{th:s_2m_2m_1_odd}, $\s_{10}(11)=25$,
which gives another example: 
$\s_8(11) - 8 = 12$, $\s_{10}(11) - 10 = 15$.

Conjecture 4.6 of \cite{Gao:2014} claims 
$\s_{2m}(d) = \beta_{[1,2m]}(d) + 2m$ for every $m$. 
(In notation of \cite{Gao:2014}, 
$\eta_{2m}(\Z_2^d) = \beta_{[1,2m]}(d) + 1$.) 
The equality holds for $m=2$ (see our Theorem~\ref{th:beta_s_4}), 
and is likely to hold for even $m$ in general. 
Our Theorem~\ref{th:s_2m_2m_1_odd} disproves this conjecture 
for odd $m \geq 3$ and $d=2m+1$. 
Indeed, it is easy to see that 
a two-dimensional binary code of length $n \geq 7$ 
and distance $n-2$ or higher can not exist. 
Therefore, $\beta_{[1,2m]}(2m+1) = N(2m+1,2m+1) < 2m+3$. 
By applying \eqref{eq:even_beta_1}, 
we get $\beta_{[1,2m]}(2m+1) = 2m+2$ 
while $s_{2m}(2m+1) = 4m+5$.

\begin{observation}\label{th:2m_2m-2}
If $\s_{2m}(d) - \s_{2m-2}(d) \geq 3$ 
then $\beta_{2m}(d) = \s_{2m}(d)-1$. 
Indeed, consider a sequence $S$ of length $\s_{2m}(d)-1$ over $\Z_2^d$ 
that does not contain a zero-sum subsequence of length $2m$. 
If $\beta_{2m}(d) < \s_{2m}(d)-1$, 
there is $z\in\Z_2^d$ which appears in $S$ at least twice. 
Remove two copies of $z$ to obtain a sequence $S'$ of length 
$\s_{2m}(d)-3$. 
As $\s_{2m}(d)-3 \geq \s_{2m-2}(d)$, 
$S'$ must contain a zero-sum subsequence of length $2m-2$. 
By adding back two copies of $z$ we get 
a zero-sum subsequence of length $2m$ in $S$.
\end{observation}

In light of Conjectures \ref{th:s_2m_2m_1_odd}, \ref{th:s_2m_2m_1_odd}, 
and Observation~\ref{th:2m_2m-2}, 
we expect $\beta_{2m}(d) = \s_{2m}(d)-1$ to hold for all odd $m$ 
and $2m-3 \leq d \leq 3m$.

\vspace{2mm}

It follows from Theorem \ref{th:s_2m_2m_1_even} 
and \eqref{eq:recursive} that 
$\s_{2m}(2m+2) \geq 4m+3$ for even $m$. 
By Corollary~\ref{th:s_lower3} and~\eqref{eq:beta_s_beta}, 
$\s_{2m}(2m+2) \geq 4m+7$ for odd $m$. 
Hence, Theorems \ref{th:R_s}, \ref{th:s_2m_2m}, 
\ref{th:s_2m_2m_1_odd}, and \ref{th:s_2m_2m_1_even} yield

\begin{corollary}\label{th:Enomoto_my_1}
\[
  R_{2m}(n) \; = \;  
  \begin{cases} 
    2m + 1, & \mbox{if }\; n = 4m+1, \\
    2m + 1, & \mbox{if }\; 4m+2 \leq n \leq 4m+4, \; m \mbox{ is odd}, \\
    2m + 2, & \mbox{if }\; 4m+5 \leq n \leq 4m+6, \; m \mbox{ is odd}, \\
    2m + 2, & \mbox{if }\; n = 4m+2, \; m \mbox{ is even}.
  \end{cases}
\]
\end{corollary}

The next statement (which was also proved in \cite{Sidorenko:2017}) 
can be easily derived from Theorem~\ref{th:R_s} 
and~\eqref{eq:Bassalygo}:

\begin{corollary}\label{th:s_d_large}
For any fixed $m$,
$\;
  \s_{2m}(d) = \Theta\left(2^{d/m}\right) 
$
as $d\rightarrow\infty\;$.
\end{corollary}

It follows from \eqref{eq:BCH} and Theorem \ref{th:R_s} that 
$\; \limsup_{d\rightarrow\infty} \s_{2m}(d) \: 2^{-d/m} \geq 1$. 
We can improve this bound for odd $m$. 
Namely, \eqref{eq:beta_s_beta} and 
Theorems \ref{th:beta_123_2m}, \ref{th:beta_24_2m}, \ref{th:beta_d-1} 
yield $s_{2m}(d) \geq 2N(d-1,2m+1) + 2$. 
Coupled with \eqref{eq:BCH}, it leads to

\begin{corollary}
If $m$ is odd, 
$\; 
  {\displaystyle\limsup_{d\rightarrow\infty}} \; \s_{2m}(d) \: 2^{-d/m} 
    \; \geq \; 2^{1-1/m} \; .
$ 
\end{corollary}

By Theorem \ref{th:beta_s_4}, 
bounds on $N(d,5)$ from \cref{sec:codes} 
translate into bounds on $\s_4(d)$. 
For $d \leq 10$, we get the exact values:

\vspace{3mm}
\begin{tabular}{cccccccccccc}
  $d$           & 1 & 2 & 3 & 4 &  5 &  6 &  7 &  8 &  9 & 10
    \\
  $\s_4(d)$      & 5 & 6 & 7 & 9 & 10 & 12 & 15 & 21 & 27 & 37
\end{tabular}
\vspace{3mm}

\section{Proofs of theorems}\label{sec:proofs}

For $A\subseteq\Z_2^d$, we denote the sum of elements of $A$ by $\sum A$. 

\begin{proof}[\bf{Proof of Theorem \ref{th:beta_123_2m}}]
First, we will prove $\beta_{[1,2m]}(d) \geq N(d,2m+1)$. 
Consider a linear code $C$ of length $n=N(d,2m+1)$, redundancy $d$ 
and distance at least $2m+1$. 
Its parity check matrix $M$ has size $d \times n$. 
Since $C$ has no words of weights $1,2,\ldots,2m$, 
the sum of any $k\in\{1,2,\ldots,2m\}$ columns of $M$ 
is not a zero vector. 
It means that the columns of $M$, 
being interpreted as $n$ vectors in $\Z_2^d$, 
form a set without zero-sum subsets of sizes $2m$ and less.

To prove $\beta_{[1,2m]}(d) \leq N(d,2m+1)$, 
consider a set $A$ of size $\beta_{[1,2m]}(d)$ in $\Z_2^d$ 
which has no zero-sum subsets of sizes $2m$ and less. 
If $A$ does not contain a basis in $\Z_2^d$, 
then there exists a vector $x \in \Z_2^d$ 
which cannot be represented as a sum of some vectors from $A$. 
Then $A \cup \{ x \}$ would have no zero-sum subsets of sizes $2m$ and less 
which contradicts with the maximality of $|A|$. 
Hence, $A$ contains a basis. 
Then a $d \times |A|$ matrix, 
whose columns are the vectors in $A$, 
has rank $d$ 
and is a parity check matrix of a code of length $n$, 
redundancy $d$ and distance at least $2m+1$.
\end{proof}

\begin{proof}[\bf{Proof of Theorem \ref{th:beta_24_2m}}]
Consider a set $A$ of size $\beta_{[1,2m]}(d)$ in $\Z_2^d$ 
which has no zero-sum subsets of sizes $1,2,\ldots,2m$. 
In particular, $0 \notin A$. 
It is obvious that $A \cup \{0\}$ 
does not have zero-sum subsets of sizes $2,4,\ldots,2m$. 
Hence, $\beta_{2[1,m]}(d) \geq \beta_{[1,2m]}(d) + 1$. 

To prove the opposite inequality, 
consider a set $B$ of size $\beta_{2[1,m]}(d)$ in $\Z_2^d$ 
which has no zero-sum subsets of sizes $2,4,\ldots,2m$. 
Select $y \in B$ and define $B_y = \{x+y \; | \; x \in B \}$. 
Notice that $B_y$ does not have 
zero-sum subsets of sizes $2,4,\ldots,2m$ 
and contains the zero vector. 
Then $B_y \backslash \{0\}$ will have no zero-sum subsets of sizes $1,2,\ldots,2m$. 
Therefore, $\beta_{[1,2m]}(d) \geq |B_y \backslash \{0\}| = \beta_{2[1,m]}(d) - 1$.
\end{proof}

\begin{lemma}\label{th:beta_d_small_1}
If $A\subset\Z_2^d$, $|A| > 2^{d-1}$ and $\sum A \neq 0$, 
then there exists $B \subset A$ such that 
$|B| = |A|-2$ and $\sum B = 0$.
\end{lemma}

\begin{proof}[\bf{Proof}]
There are exactly $2^d$ solutions $(x,y)$ 
of the equation $x+y = \sum A$ where $x,y \in \Z_2^d$. 
Let $\overline{A} = \Z_2^d \backslash A$. 
The number of solutions with $x \in \overline{A}$ or $y \in \overline{A}$ 
is at most $2|\overline{A}| = 2 (2^d - |A|) < 2^d$. 
Hence, there exists a solution $(x,y)$ where $x,y \in A$. 
As $x+y = \sum A \neq 0$, we get $x \neq y$.
Set $B = A \backslash \{x,y\}$. 
Then $|B| = |A|-2$ and $\sum B = 0$.
\end{proof}

\begin{lemma}\label{th:beta_d_small_2}
If $A\subset\Z_2^d$, $|A| \geq 2^{d-1} + 2$, 
then there exists $B \subset A$ such that 
$|B| = |A|-3$ and $\sum B = 0$.
\end{lemma}

\begin{proof}[\bf{Proof}]
As $|A| \geq 2$, there exists $x \in A$ such that $x \neq \sum A$. 
Set $A_1 = A \backslash \{x\}$, then $\sum A_1 \neq 0$. 
By Lemma~\ref{th:beta_d_small_1}, there exists $B \subset A_1$ such that
$|B| = |A|-3$ and $\sum B = 0$.
\end{proof}

\begin{proof}[\bf{Proof of Theorem \ref{th:beta_d_small}}]
We will prove the lower bound first. 
In the case $m \neq 2^{d-1} - 2$, 
select $A \subseteq \Z_2^d$ such that 
$|A|=2m+2$ and $\sum A = 0$. 
If $B \subset A$, $|B|=2m$, and $A \backslash B = \{x,y\}$, 
then $\sum B = \sum A - (x+y) = x+y \neq 0$.
Hence, $\beta_{2m}(d) \geq 2m+2$. 
In the case $m = 2^{d-1} - 2$, 
the bound $\beta_{2m}(d) \geq 2m$ is trivial.

Now we will prove the upper bound. 
If $m \neq 2^{d-1} - 2$, 
consider $A\subset\Z_2^d$ where $|A| = 2m+3 \geq 2^{d-1} + 3$. 
By Lemma~\ref{th:beta_d_small_2}, there exists $B \subset A$ 
such that $|B|=2m$ and $\sum B = 0$. 
Therefore, $\beta_{2m}(d) \leq 2m+2$. 
If $m = 2^{d-1} - 2$, 
consider $A\subset\Z_2^d$ where $|A| = 2m+1 = 2^d - 3$. 
Let $x = \sum A$, and $\Z_2^d \backslash A = \{a,b,c\}$. 
Since $x = \sum A = a+b+c = a + (b+c) \neq a$ 
(and similarly, $x \neq b,c$), 
we conclude that $x \in A$. 
Set $B = A \backslash \{x\}$. Then $|B|=2m$ and $\sum B = 0$, 
so $\beta_{2m}(d) \leq 2m$.
\end{proof}

\begin{proof}[\bf{Proof of Theorem \ref{th:beta_d-1}}]
Consider a set $A$ of size $b \leq \beta_{2[1,m]}(d-1)$ in $\Z_2^{d-1}$ 
which does not have zero-sum subsets of sizes $2,4,\ldots,2m$. 
For $i=0,1$, obtain $A_i \subset \Z_2^d$ 
by attaching $i$ to each vector of $A$ as the $d$'th entry. 
We claim that $A_0 \cup A_1$ does not have 
zero-sum subsets of size $2m$. 
Indeed, suppose $X \subseteq (A_0 \cup A_1)$, $|X|=2m$. 
Let 
$X_i = \{(z_1,z_2,\ldots,z_{d-1}) | \; (z_1,z_2,\ldots,z_{d-1},i) \in X\}$ 
and $Y = (X_0 \cup X_1) \backslash (X_0 \cap X_1)$.
Then $X_0,X_1,Y$ are subsets of $A$. 
If $\sum X = 0$ then $|X_0|, |X_1|$ are even and 
$\sum X_0 + \sum X_1 = 0$. 
Notice that $|X_0 \cap X_1|=m$ is impossible, 
as it would imply $|X_0|=|X_1|=m$, but $m$ is odd. 
Hence, $\sum Y = \sum X_0 + \sum X_1 - 2 \sum (X_0 \cap X_1) = 0$ 
while $|Y| = 2m - 2|X_0 \cap X_1|$ is even and not equal to zero. 
Selecting $b = \beta_{2[1,m]}(d-1)$ in $\Z_2^{d-1}$, we get 
$\beta_{2m}(d) \geq |A_0 \cup A_1| = 2\beta_{2[1,m]}(d-1)$.

When $b$ is even, $\sum (A_0 \cup A_1) = 0$, 
so $A_0 \cup A_1$ does not have zero-sum subsets of size $2b-2m$ either. 
Hence, $\beta_{2b-2m}(d) \geq |A_0 \cup A_1| = 2b$.
\end{proof}

\begin{proof}[\bf{Proof of Theorem \ref{th:s_max_beta}}]
We will show first that 
$\s_{2m}(d) \geq 1 + \beta_{2[j,m]}(d) + (2m-2j)$. 
Indeed, consider a set $A$ of size $\beta_{2[j,m]}(d)$ in $\Z_2^d$ 
which does not have zero-sum subsets of sizes $2j,2j+2,\ldots,2m$. 
Select $x \in A$ and form a sequence from all the elements of $A$ 
plus $(2m-2j)$ extra copies of $x$. 
It is easy to see that this sequence does not have zero-sum subsequences of size $2m$. 
Hence, 
$\s_{2m}(d) \geq 1 + \max_{1 \leq j \leq m} \left\{\beta_{2[j,m]}(d) + (2m-2j)\right\}$.
Now we need to prove the opposite inequality. 
Let $s$ be an integer such that 
$s > \beta_{2[j,m]}(d) + (2m-2j)$ 
for every $j=1,2,\ldots,m$. 
Define the product of an integer $n$ and $z\in\Z_2^d$ 
as $z$ if $n$ is odd, and $0$ if $n$ is even.
Let $f(z)$ be the number of appearances of $z\in\Z_2^d$ 
in a sequence of length $s$ over $\Z_2^d$. 
We are going to show that there exists 
a nonnegative integer function $g$ on $\Z_2^d$ 
such that $g \leq f$, 
$\;\sum_{z\in\Z_2^d} g(z) = 2m$, 
and $\sum_{z\in\Z_2^d} g(z)z = 0$. 
Indeed, set $f_1(z)=1$ if $f(z)$ is odd, 
and $f_1(z)=0$ if $f(z)$ is even. 
Set $f_2(z) = f(z) - f_1(z)$. 
All values of $f_2(z)$ are even. 
If $\sum_{z\in\Z_2^d} f_2(z) \geq 2m$, 
then there exists $g \leq f_2$ such that all values of $g$ are even and 
$\sum_{z\in\Z_2^d} g(z) = 2m$. 
Hence, we can assume that 
$\sum_{z\in\Z_2^d} f_2(z) = 2(m-j)$ where $j\in\{1,2,\ldots,m\}$. 
Let $A = \{ z\in\Z_2^d: \; f_1(z)=1 \}$. 
Since 
$|A| = \sum_{z\in\Z_2^d} (f(z)-f_2(z)) = s - 2(m-j) > \beta_{2[j,m]}(d)$, 
there exists $B \subseteq A$ 
such that $|B| = k \in \{2j,2j+2,\ldots,2m\}$ 
and $\sum_{z \in B} z = 0$. 
Set $f_B(z)=1$ if $z \in B$, and $f_B(z)=0$ otherwise. 
Choose a function $f_3 \leq f_2$ such that 
all values of $f_3$ are even and 
$\sum_{z\in\Z_2^d} f_3(z) = 2m - k$. 
Then $f_3+f_B$ is the required function $g$.
\end{proof}

\begin{proof}[\bf{Proof of Theorem \ref{th:beta_s_4}}]
By definition, $\beta_4(d)=\beta_{\{2,4\}}(d)$. 
Thus, Theorems~\ref{th:beta_123_2m} and~\ref{th:beta_24_2m} 
imply $\beta_4(d) = N(d,5)+1$, 
while Theorem~\ref{th:s_max_beta} implies $\s_4(d) = \beta_4(d) + 3$. 
\end{proof}

\begin{proof}[\bf{Proof of Theorem \ref{th:R_s}}]
We will show first that 
$n < \s_{2m}(d)$ implies $R_{2m}(n) \leq d$, 
and then show that $n \geq \s_{2m}(d-1)$ implies $R_{2m}(n) \geq d$. 
As $\s_{2m}(d-1) < \s_{2m}(d)$ (see \eqref{eq:recursive}), 
these two statements establish the theorem.

Let $n < \s_{2m}(d)$.
Then there exists a sequence $S$ of length $n$ over $\Z_2^d$ 
which does not have zero-sum subsequences of size $2m$. 
Write the vectors of $S$ as $d \times n$ binary matrix $M$. 
This matrix does not have $2m$ columns that sum up to a zero vector.
Its rank is $r \leq d$. 
Take $r$ independent rows of $M$ 
to get an $r \times n$ matrix $M_1$ of rank $r$. 
We claim that $M_1$ does not have $2m$ columns 
that sum up to a zero vector. 
Indeed, any row of $M$ that is not in $M_1$ 
is the sum of some rows of $M_1$. 
If $M_1$ had a set of $2m$ columns which sum up to a zero vector, 
then the same columns in $M$ would also sum up to a zero vector. 
Let $C$ be the linear code whose parity check matrix is $M_1$. 
This code has length $n$, redundancy $r$ 
and does not have a word of weight $2m$. 
Hence, $R_{2m}(n) \leq r \leq d$.

Let $n \geq \s_{2m}(d-1)$. 
Consider a linear code $C$ of length $n$ and redundancy $r$ 
which does not have words of weight $2m$. 
Let $M$ be an $r \times n$ parity check matrix of $C$. 
If $r \leq d-1$, then $n \geq \s_{2m}(r)$, 
and there exists a set of $2m$ columns in $M$ that sum up to a zero vector. 
It contradicts with the assumption that $C$ does not have words of weight $2m$. 
Hence, $r \geq d$. We have proved that $R_{2m}(n) \geq d$.
\end{proof}

\begin{proof}[\bf{Proof of Theorem \ref{th:s_m_small}}]
It follows from \eqref{eq:Enomoto1} 
that $R_{2m}(2m+d)=d+1$ when $d<2m$. 
Hence, by Theorem~\ref{th:R_s},  
$\s_{2m}(d) \leq 2m + d$. 
The opposite inequality is provided by~\eqref{eq:s_lower}. 
\end{proof}

\begin{observation}\label{th:operations}
The following operations do not change the fact 
whether a binary matrix has a set of $2m$ columns 
that sum up to a zero vector: 
permutations of columns (rows), 
additions of one row to another, 
additions of the same vector to each column. 
In particular, bringing a matrix to binormal form 
by Lemma~\ref{th:Enomoto2} does not change the fact whether 
such a set of $2m$ columns exists.
\end{observation} 

\begin{lemma}\label{th:rank}
Let $M$ be an $d \times n$ binary matrix whose rank is less than $d$. 
If $n \geq \s_{2m}(d-1)$, 
then $M$ contains $2m$ columns that sum up to a zero vector. 
\end{lemma}

\begin{proof}[\bf{Proof}]
Use additions of one row to another to make the last row entirely zero. 
By Observation~\ref{th:operations}, it does not change the fact 
whether there is a set of $2m$ columns that sum up to a zero vector. 
As $n > \s_{2m}(d-1)$, such a set of columns must exist. 
By Observation~\ref{th:operations}, 
the same is true for the original matrix. 
\end{proof}

\begin{proof}[\bf{Proof of Theorem \ref{th:s_2m_2m}}]
When $k=2$ and $d=2m$, \eqref{eq:s_lower2} yields 
$\s_{2m}(2m) \geq 4m+1$. 
To prove $\s_{2m}(2m) \leq 4m+1$, 
consider a sequence $S$ of size $4m+1$ over $\Z_2^{2m}$. 
We need to prove that $S$ has a zero-sum subsequence of size $2m$.
Let $x$ be the sum of all elements of $S$. 
Add $x$ to each element of $S$ 
and denote the resulting sequence by $S_0$. 
If $S_0$ has a zero-sum subsequence of size $2m$, then $S$ has it, too. 

Consider a $(2m) \times (4m+1)$ binary matrix $M$ 
whose columns represent the vectors from $S_0$. 
By Theorem~\ref{th:s_m_small}, $4m+1 \geq \s_{2m}(2m-1)$. 
If the rank of $M$ is less than $2m$, 
then by Lemma~\ref{th:rank}, it has a set of $2m$ columns 
that sum up to a zero vector. 
Suppose that the rank of $M$ is $2m$. 
Since the sum of all vectors of $S_0$ is a zero vector, 
$M$ satisfies the conditions of Corollary~\ref{th:Enomoto}. 
Thus, there are $2m$ columns in $M$ whose sum is a zero vector.
\end{proof}

\begin{lemma}\label{th:2m_1_2m_1}
Let $C=[c_{ij}]$ be a $t \times t$ binary matrix.
If $C$ does not have a row where all off-diagonal entries are equal,
then it has three distinct rows $i,j,k$ such that 
$c_{ij}+c_{ik}=1$ and $c_{ji}+c_{jk}=1$. 
\end{lemma}

\begin{proof}[\bf{Proof}]
Let $a_i$ be the number of off-diagonal nonzero entries in row $i$, 
and $b_i = (t-1) - a_i$ be the number of off-diagonal zero entries. 
Let 
$a = \min\{ a_1,a_2,\ldots,a_t \}$\!, 
$b = \min\{ b_1,b_2,\ldots,b_t \}$. 
Since the statement of the lemma holds for matrix $[c_{ij}]$ 
if and only if it holds for matrix $[1-c_{ij}]$, 
we may assume that $a \leq b$. 
Let index $i$ be such that $a_i=a$. 
Since $a_i > 0$, there exists index $j \neq i$ such that $c_{ij}=1$. 
Consider two cases for $c_{ji}$. 

0). $c_{ji}=0$. 
As $a_j > a_i - 1$, there exists $k \neq i,j$ such that 
$c_{ik}=0$ and $c_{jk}=1$.

1). $c_{ji}=1$. 
As $a_i + a_j = a_i + (t-1 - b_j) \leq (t-1) + a_i - b < t$, 
there exists $k \neq i,j$ such that $c_{ik} = c_{jk} = 0$.
\end{proof}

Let $M$ be an $r \times c$ matrix, 
and $1 \leq r' \leq r'' \leq r$, $\: 1 \leq c' \leq c'' \leq c$. 
We denote by $M[r':r'',\: c':c'']$ the submatrix of $M$ formed by 
rows $r', r'+1, \ldots, r''$ and columns $c', c'+1, \ldots, c''$. 

\begin{lemma}\label{th:2m_1_4m_5}
Let $k$ be even, 
and $M$ be a $(k+1) \times (2k+5)$ matrix in binormal form. 
One can pick up $k$ columns in $M$ whose sum is 
the $(k+1)$-dimensional zero vector.
\end{lemma}

\begin{proof}[\bf{Proof}]
Let $M=[a_{ij}]$. 
Let $C=[c_{ij}]$ be a $(k+1) \times (k+1)$ binary matrix 
where $c_{ij} = a_{i,2j-1} = a_{i,2j}$ for $i \neq j$. 
The values of diagonal entries $c_{ii}$ are not important 
and may be set arbitrarily. 
Suppose, there is a row in $C$ where 
the sum of off-diagonal entries is $0$. 
Without loss of generality, we may assume 
that it is the last row, so
$\sum_{i=1}^{k} c_{k+1,i} = 0$. 
Since $M[1:k,1:2k]$ is in binormal form, 
by Lemma~\ref{th:Enomoto1}, 
there is a set of $k$ indices $j(i) \in \{ 2i-1,2i \}$ 
$\; (i=1,2,\ldots,k)$ 
such that the sum of columns $j(1),j(2),\ldots,j(k)$ 
in $M[1:k,1:2k]$ is the $k$-dimensional zero vector. 
Since $a_{k+1,j(i)} = c_{k+1,i}$, 
the condition $\sum_{i=1}^{k} c_{k+1,i} = 0$ 
guarantees that the sum of columns $j(1),j(2),\ldots,j(k)$ in $M$ 
is the $(k+1)$-dimensional zero vector. 
Thus, we may assume that the sum of off-diagonal entries 
in each row of $C$ is equal to $1$. 
It means that $C$ satisfies the conditions of Lemma~\ref{th:2m_1_2m_1}. 
Hence, without loss of generality, we may assume that 
$c_{k-1,k} + c_{k-1,k+1} = 1$ and 
$c_{k,k-1} + c_{k,k+1} = 1$. 
Since the sum of off-diagonal entries in any row of $C$ is $1$, we get 
$\sum_{j=1}^{k-2} c_{ij} = 0$ for $i=k-1,k$. 
Let $x = \sum_{j=1}^{k-2} c_{k+1,j}$. 
Then the sum of all columns 
in the $3 \times (k-2)$ matrix $C[k-1:k+1,\: 1:k-2]$ 
is equal to $(0,0,x)^T$. 
We claim that the $3 \times 9$ matrix $M[k-1:k+1,\: 2k-3:2k+5]$ 
has two columns whose sum is $(0,0,x)^T$. 
Indeed, among $9$ columns there must be two equal, their sum is $(0,0,0)^T$. 
Since $M$ is in binormal form, 
its columns $(2k+1)$ and $(2k+2)$ differ only in the last entry. 
Hence, $M[k-1:k+1, 2k-3:2k+5]$ 
contains two rows whose sum is $(0,0,1)^T$. 
Now we select two columns in $M[k-1:k+1,\: 2k-3:2k+5]$ 
whose sum is $(0,0,x)^T$ 
and call the columns of $M$ which contain them {\it special}. 
Let $(y_1,y_2,\ldots,y_{k+1})^T$ be the sum of the two special columns. 
We already know that $y_{k-1}=y_{k}=0$ and $y_{k+1}=x$. 
Since $M[1:k-2,\:1:2k-4]$ is in binormal form, 
by Lemma~\ref{th:Enomoto1}, 
there is a set of $k-2$ indices $l(i) \in \{ 2i-1,2i \}$ 
$\; (i=1,2,\ldots,k-2)$ 
such that the sum of columns $l(1),l(2),\ldots,l(k-2)$ 
in $M[1:k-2,\:1:2k-4]$ is equal to $(y_1,y_2,\ldots,y_{k-2})^T$. 
Then the sum of these $k-2$ columns plus the two special columns in $M$ 
is the $(k+1)$-dimensional zero vector.
\end{proof}

\begin{proof}[\bf{Proof of Theorem \ref{th:s_2m_2m_1_odd}}]
By \eqref{eq:beta_s_beta} and Corollary~\ref{th:s_lower3}, 
$\s_{2m}(2m+1) \geq \linebreak \beta_{2m}(2m+1) + 1 \geq 4m+5$. 
To prove $\s_{2m}(2m+1) \leq 4m+5$, 
consider a sequence $S$ of size $4m+5$ over $\Z_2^{2m+1}$. 
We need to prove that $S$ has a zero-sum subsequence of size $2m$.
Let $x$ be the sum of all elements of $S$. 
Add $x$ to each element of $S$ and denote the resulting sequence by $S_0$. 
If $S_0$ has a zero-sum subsequence of size $2m$, then $S$ has it, too. 

Consider a $(2m+1) \times (4m+5)$ binary matrix $M$ 
whose columns represent the vectors from $S_0$. 
By Theorem~\ref{th:s_2m_2m}, $4m+5 \geq \s_{2m}(2m)$. 
If the rank of $M$ is less than $2m+1$, 
then by Lemma~\ref{th:rank}, it has a set of $2m$ columns 
that sum up to a zero vector. 

Suppose that the rank of $M$ is $2m+1$. 
Since the sum of all vectors of $S_0$ is a zero vector, 
$M$ satisfies the conditions of Lemma~\ref{th:Enomoto2} 
and can be brought to matrix $M'$ in binormal form 
by permutations of columns and additions of one row to another.
By Lemma~\ref{th:2m_1_4m_5}, $M'$ has a set of $2m$ columns 
that sum up to a zero vector. 
Then by Observation~\ref{th:operations}, $M$ has such a set, too.
\end{proof}

\begin{proof}[\bf{Proof of Theorem \ref{th:beta_odd}}]
By \eqref{eq:beta_s_beta}, $\beta_{2m}(d) \leq \s_{2m}(d) - 1$. 
It remains to prove $\beta_{2m}(d) \geq \s_{2m}(d) - 1$.
The values of $\s_{2m}(d)$ for $d \leq 2m+1$ were determined in 
Theorems~\ref{th:s_m_small}, \ref{th:s_2m_2m}, 
and~\ref{th:s_2m_2m_1_odd}. 
Theorems~\ref{th:beta_24_2m} and~\ref{th:beta_d-1} yield 
$\beta_{2m}(d) \geq 2\beta_{[1,2m]}(d-1) + 2$. 
The set of $d-1$ basis vectors in $\Z_2^{d-1}$ demonstrates that 
$\beta_{[1,2m]}(d-1) \geq d-1$, so we get 
\[
  \beta_{2m}(2m-1) \; \geq \;
  2(2m-2)+2 \; = \; 4m-2 \; = \; \s_{2m}(2m-1)-1 , 
\]
\[
  \beta_{2m}(2m  ) \; \geq \; 
  2(2m-1)+2 \; = \; 4m   \; = \; \s_{2m}(2m  )-1 . 
\]
By \eqref{eq:even_beta_1}, $\beta_{[1,2m]}(2m) \geq 2m+1$, 
and hence, 
\[ 
  \beta_{2m}(2m+1) \; \geq \; 
  2(2m+1)+2 \; = \; 4m+4 \; = \; \s_{2m}(2m+1)-1 . 
\] 
We use Theorem~\ref{th:beta_d-1} 
with odd $m \geq 1$ and $b = 2m+2 \leq \beta_{2[1,m]}(2m)$ 
to get 
$\beta_{2m+4}(2m+1) = \beta_{2b-2m}(2m+1) \geq 2b = 4m+4$ 
for $m \geq 1$. 
Substituting $m-2 \geq 1$ instead of $m$, we get
\[
  \beta_{2m}(2m-3) \; \geq \; 4m-4 \; = \; \s_{2m}(2m-3)-1 
  \;\;\;\mbox{for}\;\; m \geq 3 .
\] 
Finally, 
\[ 
  \beta_{2m}(2m-2) \; \geq \;
  \beta_{2m}(2m-3) + 1 \; \geq \; 4m-3 \; = \; \s_{2m}(2m-2)-1  
  \;\;\;\mbox{for}\;\; m \geq 3 .
\] 
\end{proof}

\begin{lemma}\label{th:code_without_2_and_4}
Let $C$ be a linear binary code of length $n \geq 10$. 
If $C$ does not have words of Hamming weight $2$ and $4$, 
then its dual code $C^\perp$ has a nonzero word of weight $l$ 
where $|l - n/2| \geq 2$.
\end{lemma}

\begin{proof}[\bf{Proof}]
Suppose, to the contrary, that the weights of nonzero words of $C^\perp$ 
lay in the interval $[(n-3)/2,\: (n+3)/2]$.
Let 
$D_n = \{ 0,\: (n-3)/2,\: (n-1)/2,\: (n+1)/2,\: (n+3)/2 \}$ 
if $n$ is odd, and 
$D_n = \{ 0,\: (n-2)/2,\: n/2,\: (n+2)/2 \}$ 
if $n$ is even. 
Let $r$ be the dimension of $C^\perp$, 
so the dimension of $C$ is $k=n-r$. 
Let $A_j\;$ ($B_j$) 
denote the number of words of weight $j$ 
in $C\;$ ($C^\perp$). 
Then $A_0=1$, $A_2=A_4=0$, $B_0=1$, $B_j=0$ for $j \notin D_n$, 
and $\sum_{j \in D_n} B_j = 2^r$. 
Consider a linear combination of MacWilliams identities \eqref{eq:MWI} 
with $\lambda=1,2,3,4$, 
where coefficients $c_\lambda$ will be selected later:
\[
  \sum_{\lambda=1}^4 c_\lambda \;
  2^r \sum_{j=0}^\lambda \binom{n-j}{\lambda-j} A_j
    \; = \;
  \sum_{\lambda=1}^4 c_\lambda \;
  2^\lambda \sum_{j=0}^{n-\lambda} \binom{n-j}{\lambda} B_j
    \; .
\]
As $D_n \subseteq [0,n - \lambda]$ for $n \geq 10$, $\lambda \leq 4$, 
and $B_j=0$ for $j \notin D_n$, 
we can rewrite it as
\begin{equation}\label{eq:MWI1}
  2^r \sum_{\lambda=1}^4 c_\lambda 
    \sum_{j=0}^\lambda \binom{n-j}{\lambda-j} A_j
      \; = \;
  \sum_{j \in D_n} f_j B_j
      \; ,
\end{equation}
where 
\[
  f_j \; = \; 
  \sum_{\lambda=1}^4 c_\lambda \; 2^\lambda 
    \sum_{j=0}^{n-\lambda} \binom{n-j}{\lambda}
      \; .
\]
As $A_2=A_4=0$, the left hand side of \eqref{eq:MWI1} 
is a linear combination of $A_0,A_1,A_3$. 
We are going to choose coefficients $c_1,c_2,c_3,c_4$ in such a way 
that $A_1$ and $A_3$ are eliminated 
while the values of $f_j$ are equal for all $j \in D_n \backslash \{0\}$. 

If $n$ is even, set $c_1=4n(n-1)(n-2)$, $c_2=-12(n-2)^2$, 
$c_3=24(n-3)$, $c_4=-24$. 
In this case, we get $f_0=0$, 
$f_j = n^2(n+2)(n-2)$ for $j \in D_n \backslash \{0\}$, 
and \eqref{eq:MWI1} is reduced to 
\[
  2^r n(n-1)(n-2)(n+3) A_0 \; = \; 
  n^2(n+2)(n-2) \sum_{j \in D_n \backslash \{0\}} B_j \; .
\]
Since $A_0=1$ and $\sum_{j \in D_n \backslash \{0\}} B_j = 2^r - 1$, 
we can simplify it further to 
\[
  n(n-2)(3 \cdot 2^r - n(n+2)) \; = \; 0 \; ,
\]
which has no integer solutions for $n>6$. 

If $n$ is odd, set $c_1=4(n+1)(n-1)(n-3)$, $c_2=-12(n-1)(n-3)$, 
$c_3=24(n-3)$, $c_4=-24$. 
In this case, we get $f_0=0$, 
$f_j = (n+3)(n+1)(n-1)(n-3)$ for $j \in D_n \backslash \{0\}$, 
and \eqref{eq:MWI1} is reduced to 
\[
  2^r n(n-1)(n-3)(n+4) A_0 \; = \; 
  (n+3)(n+1)(n-1)(n-3) \sum_{j \in D_n \backslash \{0\}} B_j \; .
\]
Since $A_0=1$ and $\sum_{j \in D_n \backslash \{0\}} B_j = 2^r - 1$, 
we can simplify it further to 
\[
  (n-1)(n-3)(3 \cdot 2^r - (n+1)(n+3)) \; = \; 0 \; ,
\]
which has no integer solutions for $n>3$. 
\end{proof}

\begin{lemma}\label{th:beta_2_beta}
$\;
 \beta_{2[1,m]}(d) \; \leq \; \max\{9,\; 2 \beta_{2[1,m]}(d-1) - 4\} \; .
$
\end{lemma}

\begin{proof}[\bf{Proof}]
Suppose, $\beta_{2[1,m]}(d) \geq 10$. 
We need to prove 
$\beta_{2[1,m]}(d) \leq 2 \beta_{2[1,m]}(d-1) - 4$. 
Consider a set $A$ of size $n=\beta_{2[1,m]}(d)$ in $\Z_2^d$ 
which does not have zero-sum subsets of sizes $2,4,\ldots,2m$. 
Write $n$ vectors of $A$ column-wise 
as a $d \times n$ binary matrix $M$. 
Similarly to the proof of Theorem~\ref{th:beta_123_2m}, 
the maximality of $|A|$ ensures that $M$ has rank $d$. 
Let $C$ be the linear code of length $n$ whose parity check matrix is $M$. 
This code does not have words of weight $2,4,\ldots,2m$. 
As $n \geq 10$, by Lemma~\ref{th:code_without_2_and_4}, 
the dual code $C^\perp$ has a word of weight $l$ 
where $|l - n/2| \geq 2$. 
Then there exists a parity check matrix $M_1$ of code $C$ 
such that this word is the first row of $M_1$. 
Notice that the sum of any $2k$ columns of $M_1$ is not a zero vector 
($k=1,2,\ldots,m$). 
If $l \geq (n+4)/2$, 
remove from $M_1$ all columns which contain $0$ in the first row. 
If $l \leq (n-4)/2$, 
remove from $M_1$ all columns which contain $1$ in the first row. 
The resulting matrix $M_2$ is of size $d \times t$ 
where $t \geq (n+4)/2$. 
All entries in the first row of $M_2$ are equal. 
Remove the first row to get matrix $M_3$ of size $(d-1) \times t$. 
As $M_2$ does not have sets of columns of size $2k\;$ ($k=1,2,\ldots,m$) 
which sum up to a zero vector, 
the same is true for $M_3$. 
Therefore, 
$\beta_{2[1,m]}(d-1) \geq t \geq (n+4)/2 =  (\beta_{2[1,m]}(d) + 4) / 2$.
\end{proof}

\begin{proof}[\bf{Proof of Theorem \ref{th:beta_s_6}}]
For $3 \leq d \leq 7$, the statement of the theorem follows from 
Theorem~\ref{th:beta_odd}. 
If $d \geq 8$, 
\eqref{eq:even_beta_1} and Theorem~\ref{th:beta_24_2m} yield 
$\beta_{\{2,4,6\}}(d) \geq 10$. 
We may apply Lemma~\ref{th:beta_2_beta} to get 
$\beta_{\{2,4,6\}}(d) \leq 2 \beta_{\{2,4,6\}}(d-1) - 4$. 
By Theorem~\ref{th:beta_d-1}, 
$\beta_6(d) \geq 2 \beta_{\{2,4,6\}}(d-1)
 \geq \beta_{\{2,4,6\}}(d) + 4$. 
By definition, $\beta_{\{4,6\}}(d) = \beta_{\{2,4,6\}}(d)$, 
hence, Theorem~\ref{th:s_max_beta} yields 
$
  \s_6(d) = 1+ \max\{\beta_{\{2,4,6\}}(d)+4,\; \beta_6(d)\}
          = 1 + \beta_6(d)
$. 
\end{proof}

\begin{lemma}\label{th:digraph}
Let $D$ be a digraph with $n \equiv 1 \pmod{4}$ vertices 
where every vertex has odd out-degree. 
Then one can find in $D$ either $3$ vertices 
that span a subgraph with $2$ vertices of odd out-degree, 
or $5$ vertices 
that span a subgraph with $3$ vertices of odd out-degree.
\end{lemma}

\begin{proof}[\bf{Proof}]
Let the vertex set of $D$ be $\{1,2,\ldots,n\}$. 
Let $C=[c_{ij}]$ be the adjacency matrix: 
$c_{ij}=1$ if the arc $(i,j)$ is present in $D$,
otherwise, $c_{ij}=0\:$ ($i \neq j)$. 
The diagonal entries $c_{ii}$ are zeros. 

For a subset $A$ of vertices we define the {\it type} of $A$ 
as the number of vertices of odd out-degree 
in the subgraph of $D$ spanned by $A$. 
We need to show that there exists either 
a triple of type $2$ or a quintuple of type $3$. 
For a triple $\{i,j,k\}$, let $t(i,j,k)$ denote its type.
Similarly, for a quintuple $\{i,j,k,l,m\}$, 
its type is denoted by $t(i,j,k,l,m)$. 
Suppose, $D$ contains no triple of type $2$, 
so every triple of even type must have type $0$. 
We are going to show that there is a quintuple of type $3$.

The size of $C$ is odd, 
and the sum of the entries in each row is odd.
Hence, the sum of all entries of $C$ is odd. 
In the sum of expressions 
$f(i,j,k) = c_{ij} + c_{ji} + c_{ik} + c_{ki} + c_{jk} + c_{kj}$ 
over all triples $\{i,j,k\}$, 
each off-diagonal entry of $C$ appears $n-2 \equiv 1 \pmod{2}$ times. 
Notices, that $t(i,j,k)$ is odd (even) 
when $f(i,j,k)$ is odd (even). 
Hence, the sum of the types of all triples is odd.
As $n \equiv 1 \pmod{4}$, 
the number of all triples, $\binom{n}{3}$, is even. 
It means that there exists at least one triple of even type, 
and at least one triple of odd type. 
Then we can find a triple of even type 
and a triple of odd type that share a pair.
Notice that the expression $f(i,j,k)+f(i,j,l)+f(i,j,l)+f(j,k,l)$ 
is always even, because each arc of the subgraph 
spanned by $\{i,j,k,l\}$ is counted there twice. 
Thus, the sum of types of the $4$ triples within one quadruple is even. 
In particular, a quadruple which contains 
a triple of even type and a triple of odd type 
must contain two triples of even type and two triples of odd type. 
Without limiting generality, we can assume that 
triples $\{1,2,3\}$, $\{2,3,4\}$ are of type $0$, 
and $\{1,2,4\}$, $\{1,3,4\}$ are of odd type. 
Hence, 
\begin{equation*}
  c_{12}=c_{13}, \;\;\;
  c_{42}=c_{43}, \;\;\;
  c_{21}=c_{23}=c_{24}, \;\;\;
  c_{31}=c_{32}=c_{34}, \;\;\;
  c_{14} \neq c_{41}. 
\end{equation*}

Notice that the following operation on $D$ 
preserves the types of all subsets of odd sizes 
(in particular of sizes $3,5,n$): 
for a given vertex $i$, 
remove all arcs $(i,j)$ that were present in $D$,
and add all arcs $(i,j)$ that were not present. 
It is equivalent to replacing each off-diagonal entry $c_{ij}$ 
in row $i$ of $C$ with $1-c_{ij}$. 
We apply this {\it switching} operation, 
if necessary, to some of vertices $1,2,3,4$ to ensure 
\begin{equation}\label{eq:ccc2}
c_{12}=c_{13}=c_{42}=c_{43}=c_{21}=c_{23}=c_{24}=c_{31}=c_{32}=c_{34}=0.
\end{equation}
As indices $1$ and $4$ appear 
in \eqref{eq:ccc2} symmetrically, 
and $c_{14} \neq c_{41}$, 
we can assume, without limiting generality, 
that $c_{14}=1$ and $c_{41}=0$. 
Define
\begin{multline*}
  V_1 = \{ i\in\{4,5,\ldots,n\}: \;
            t(i,2,3)=0, \; \\
            t(i,1,2) \equiv 1 \; ({\rm mod}\: 2), \; 
            t(i,1,3) \equiv 1 \; ({\rm mod}\: 2)
        \} \; ,
\end{multline*}
\begin{multline*}
  V_2 = \{ i\in\{4,5,\ldots,n\}: \;
            t(i,1,3)=0, \; \\
            t(i,2,3) \equiv 1 \; ({\rm mod}\: 2), \; 
            t(i,1,2) \equiv 1 \; ({\rm mod}\: 2)
        \} \; ,
\end{multline*}
\begin{multline*}
  V_3 = \{ i\in\{4,5,\ldots,n\}: \;
            t(i,1,2)=0, \; \\
            t(i,1,3) \equiv 1 \; ({\rm mod}\: 2), \; 
            t(i,2,3) \equiv 1 \; ({\rm mod}\: 2)
        \} \; ,
\end{multline*}
\[
  W_0 = \{ i\in\{4,5,\ldots,n\}: \;
            t(i,1,2) = t(i,1,3) = t(i,2,3) = 0 \} \; ,
\]
\[
  W_1 = V_1 \cup \{1\} \; , \;\;\;\;\;\;
  W_2 = V_2 \cup \{2\} \; , \;\;\;\;\;\;
  W_3 = V_3 \cup \{3\} \; .
\]
As the sum of types of all triples within one quadruple is even, 
$\{W_0,W_1,W_2,W_3\}$ form a partition of $\{1,2,\ldots,n\}$. 
Notice that $4 \in V_1$. 

We claim that $t(i,j,3)=0$ for any $i \in V_1$ and $j \in V_2$ 
(if $V_2$ is empty, this statement becomes trivial).
Indeed, by the definition of $V_1$ and $V_2$, 
$t(1,2,3)=t(i,2,3)=t(1,j,3)=0$ 
while $t(i,1,2)$, $t(i,1,3)$, $t(j,2,3)$, $t(j,1,2)$ are odd. 
Suppose, $t(i,j,3)$ is odd. Then 
$t(i,j,1) \equiv t(i,j,3) + t(i,1,3) + t(j,1,3) \equiv 0 \pmod{2}$ 
and 
$t(i,j,2) \equiv t(i,j,3) + t(i,2,3) + t(j,2,3) \equiv 0 \pmod{2}$, 
which means that there are $5$ triples of type $0$ in cyclic pattern: 
\[
  t(i,j,1) = t(j,1,3) = t(1,3,2) = t(3,2,i) = t(2,i,j) = 0.
\]
If so, the principal minor of $C$, 
formed by rows and columns $i,j,1,2,3$, 
must carry equal off-diagonal entries within each row. 
But then we would have $t(i,j,3)=0$ 
which contradicts with the initial assumption 
$t(i,j,3) \equiv 1 \pmod{2}$. 
Therefore, $t(i,j,3)=0$. 

Similarly, $t(i,2,k)=0$ for any $i \in V_1$, $k \in V_3$. 

For every $i \in (W_0 \cup W_1 \backslash \{1,4\})$, 
we have $t(i,2,3)=0$, hence 
$c_{2i}=c_{23}=0$, 
$c_{3i}=c_{32}=0$, 
and $c_{i2}=c_{i3}$. 
If $c_{i2}=1$, apply switching operation to vertex $i$ 
to ensure $c_{i2}=c_{i3}=0$. 
Now we have $c_{2i}=c_{23}=c_{3i}=c_{32}=c_{i2}=c_{i3}=0$ 
for all $i \in W_0 \cup W_1$. 

For any $i \in V_1$ and $j \in V_2$, 
we have $t(i,j,3)=0$, hence $c_{ij}=c_{i3}=0$.
The same is true when $i=1$ or $j=2$. 
Therefore, $c_{ij}=0$ for all $i \in W_1$ and $j \in W_2$, 
and similarly, $c_{ik}=0$ for all $i \in W_1$ and $k \in W_3$. 

Denote by $F$ the subgraph of $D$ spanned by $W_1$. 
If $i,j \in W_1$, $i \neq j$, and $c_{ij}=c_{ji}=1$, 
then $t(i,j,2)=2$ 
(since $c_{i2}=c_{2i}=c_{j2}=c_{2j}=0$). 
Hence, $F$ does not have a pair of opposite arcs. 

We know that $c_{14}=1$, $c_{41}=0$, 
so $F$ contains a transitive tournament of size $2$ 
spanned by vertices $1$ and $4$, 
and $1$ is the vertex of zero in-degree in this tournament. 
Let $T$ be a transitive tournament of the largest possible size 
contained in $F$ such that $1$ is its vertex of zero in-degree. 
Let $i$ denote the vertex of zero out-degree in $T$. 
As the out-degree of $i$ in the whole $D$ is odd, 
there exists vertex $j$, distinct from $i$ and $1$, such that $c_{ij}=1$.
As $c_{ij}=1$, $\;j$ cannot belong to $W_2$ or $W_3$. 
The two remaining cases are $j \in W_1$ and $j \in W_0$. 

If $j \in W_1$, then $c_{ji}=0$. 
By the maximality of $T$, 
there exists a vertex $k$ in $T$ such that $c_{kj}=0$. 
As $i$ is the zero out-degree vertex in tournament $T$, 
we get $c_{ki}=1$ and $c_{ik}=0$. 
If $c_{jk}=0$, we get $t(k,i,j)=2$. 
If $c_{jk}=1$, we get $t(k,i,j,2,3)=3$. 

If $j \in W_0$, then $t(1,2,j)=0$. 
Hence, $c_{1j}=c_{12}=0$ and $c_{j1}=c_{j2}=0$. 
As $1$ is the zero in-degree vertex in tournament $T$, 
we get $c_{1i}=1$ and $c_{i1}=0$. 
If $c_{ji}=0$, we get $t(1,i,j)=2$. 
If $c_{ji}=1$, we get $t(1,i,j,2,3)=3$. 
\end{proof}

\begin{lemma}\label{th:2m_1_4m_2}
Let $m$ be even, 
and $M$ be a $(2m+1) \times (4m+2)$ matrix in binormal form. 
One can pick up $2m$ columns in $M$ whose sum is 
the $(2m+1)$-dimensional zero vector.
\end{lemma}

\begin{proof}[\bf{Proof}]
Let $n=2m+1$, $M=[a_{ij}]$. 
Let $C=[c_{ij}]$ be an $n \times n$ binary matrix 
where $c_{ij} = a_{i,2j-1} = a_{i,2j}$ for $i \neq j$. 
The values of diagonal entries $c_{ii}$ are not important 
and may be set to zero. 
For a subset $I \subseteq \{1,2,\ldots,n\}$ and $i \in I$, 
we denote $\sigma_i(I) = \sum_{j \in I \backslash \{i\}} c_{ij}$. 
We say that $I$ is of {\it type} $t$ 
if among $|I|$ values $\sigma_i(I)$ with $i \in I$ 
there are exactly $t$ that are equal to $1$. 
Let $t(I)$ denote the type of $I$.

Similarly to the proof of Lemma~\ref{th:2m_1_4m_5}, 
we can assume that the sum of off-diagonal entries 
in each row of $C$ is equal to $1$. 
Hence, $t(\{1,2,\ldots,n\})=n$.

We claim that if there exists $I \subseteq \{1,2,\ldots,n\}$ 
such that $|I| = 2t(I)-1$, 
then one can pick up $2m$ columns in $M$ with zero sum.
Indeed, as $t(1,2,\ldots,n)=n$, we have $|I| \neq n$, so $|I| \leq n-2$. 
Without limiting generality, we may assume that 
$I=\{n-2t+2,n-2t+3,\ldots,n\}$ where $t=t(I)$, $2t \leq n-1$, 
$\sigma_i(I)=1$ for $n-2t+2 \leq i \leq n-t+1$, and
$\sigma_i(I)=0$ for $n-t+2 \leq i \leq n$. 
As $M[1:n-2t+1,\: 1:2(n-2t+1)]$ is in binormal form, 
by Lemma~\ref{th:Enomoto1}, 
there is a set of indices 
$j(i)\in\{2i-1,2i\}\:$ ($i=1,2,\ldots,n-2t+1$) such that 
the sum of columns $j(1),j(2),\ldots,j(n-2t+1)$ 
in $M[1:n-2t+1,\: 1:2(n-2t+1)]$
is equal to the $(n-2t+1)$-dimensional zero vector. 
As ${\displaystyle \sum_{j=1}^{n-2t+1} c_{ij} = 1 - \sigma_i(I)}$ 
for $n-2t+2 \leq i \leq n$, 
the sum of columns $j(1),j(2),\ldots,j(n-2t+1)$ in $M$ 
has the last $t-1$ entries equal to $1$, 
and the rest equal to $0$. 
Columns $2r-1$ and $2r$ in $M$ differ only in the $r$'th row, 
so the sum of columns $2n-2t+3,2n-2t+4,\ldots,2n$ in $M$ 
also has the last $t-1$ entries equal to $1$, 
and the rest equal to $0$. 
Set $j(i)=n+1+i$ for $n-2t+2 \leq i \leq n-1$. 
Then the sum of columns $j(1),j(2),\ldots,j(n-1)$ in $M$ 
is the $n$-dimensional zero vector.

Let $D$ be a digraph with vertices $1,2,\ldots,n$ 
where an arc $(i,j)$ is present if and only if $c_{ij}=1$. 
As $n \equiv 1 \pmod{4}$, and every vertex has odd out-degree, 
$D$ satisfies conditions of Lemma~\ref{th:digraph}. 
Thus, there is a subset $I \subset \{1,2,\ldots,n\}$ 
(a triple or a quintuple) such that $|I| = 2t(I)-1$. 
As we just have shown, 
it guarantees the existence of $2m$ columns in $M$ with zero sum.
\end{proof}

\begin{lemma}\label{th:2m_1_4m_3}
Let $m$ be even, 
and $M$ be a $(2m+1) \times (4m+3)$ matrix in binormal form 
where the sum of all columns is a zero vector. 
If $M$ has two identical columns, 
then it also has a set of $2m$ columns with zero sum 
that includes at most one of the two identical columns.
\end{lemma}

\begin{proof}
Set $n=2m+1$. 
Let $x_i = (a_{1i},a_{2i},\ldots,a_{ni})^T$ 
be the $i$th column of $M\:$ ($i=1,2,\ldots,2n+1$). 
As $M$ is in binormal form, 
we can define $n \times n$ matrix $C=[c_{ij}]$ 
where $c_{ij} = a_{i,2j-1} = a_{i,2j}$ for $i \neq j$. 
The values of diagonal entries $c_{ii}$ are not important 
and may be set arbitrarily. 

If one of the two columns that are identical is the last column, 
then the statement of the lemma follows from Lemma~\ref{th:2m_1_4m_2}. 
Hence, without limiting generality, we may assume that 
the two identical columns are $2n-2$ and $2n$. 

Suppose, $c_{n,n-1} = 1$. 
Then $a_{n,2n-2} = 1$. 
As $x_{2n}=x_{2n-2}$, we get $a_{n,2n} = 1$ and $a_{n,2n-1} = 0$. 
As $M$ is in binormal form, 
$x_1+x_2+\ldots+x_{2n} = (1,1,\ldots,1)^T$.
As the sum of all columns of $M$ is a zero vector, 
$x_{2n+1} = (1,1,\ldots,1)^T$. 
Add row $n$ to each row $i\in\{1,2,\dots,n-1\}$ 
where $a_{i,2n-1} \neq a_{i,2n+1}$. 
After this is done, 
columns $2n-1$ and $2n+1$ differ only in the last entry. 
Now swap columns $2n$ and $2n+1$. 
The resulting matrix $M'$ is in binormal form. 
By Lemma~\ref{th:2m_1_4m_2}, 
there is a set of $2m=n-1$ columns in $M'$ with zero sum 
that does not include column $2n+1$ of $M'$ 
(which originated from column $2n$ of $M$). 
In this case, $M$ has a set of $2m$ columns with zero sum 
that does not include column $2n$. 
Hence, we may assume $c_{n,n-1} = 0$, 
and similarly, $c_{n-1,n} = 0$. 

The $(n-1) \times (2n-2)$ submatrix $M[1:n-1,\: 1:2n-2]$ 
is in binormal form. 
By Lemma~\ref{th:Enomoto1}, 
there is a set of indices 
$j(i) \in \{2i-1,2i\}\:$ ($i=1,2,\ldots,n-1$) 
such that the sum of columns $j(1),j(2),\ldots,j(n-1)$ 
in $M[1:n-1,\: 1:2n-2]$ is the $(n-1)$-dimensional zero vector. 
We recall that $c_{n,n-1} = 0$. 
If $\sum_{j=1}^{n-2} c_{n,j} = 0$, 
then the sum of columns $j(1),j(2),\ldots,j(n-1)$ in $M$ 
is the $n$-dimensional zero vector. 
Hence, we may assume $\sum_{j=1}^{n-2} c_{n,j} = 1$, 
and similarly, $\sum_{j=1}^{n-2} c_{n-1,j} = 1$.

The $(n-2) \times (2n-4)$ submatrix $M[1:n-2,\: 1:2n-4]$ 
is in binormal form. 
By Lemma~\ref{th:Enomoto1}, 
there is a set of indices 
$j(i) \in \{2i-1,2i\}\:$ ($i=1,2,\ldots,n-2$) 
such that $x_{j(1)} + x_{j(2)} + \ldots + x_{j(n-2)}$ 
has the first $n-2$ entries equal to $1$. 
Since
$\sum_{j=1}^{n-2} c_{n-1,j} = 1$ and 
$\sum_{j=1}^{n-2} c_{n,j} = 1$, 
the last two entries of $x_{j(1)} + x_{j(2)} + \ldots + x_{j(n-2)}$ 
are also equal to $1$. 
As $x_{2n+1}=(1,1,\ldots,1)^T$, we get 
$x_{j(1)} + x_{j(2)} + \ldots + x_{j(n-2)} + x_{2n+1} = 0$.
\end{proof}

\begin{proof}[\bf{Proof of Theorem \ref{th:s_2m_2m_1_even}}]
By \eqref{eq:recursive} and Theorem~\ref{th:s_2m_2m}, 
$\s_{2m}(2m+1) \geq \linebreak \s_{2m}(2m) + 1 = 4m+2$. 
To prove $\s_{2m}(2m+1) \leq 4m+2$, 
consider a sequence $S$ of size $4m+2$ over $\Z_2^{2m+1}$. 
Let $M$ be a $(2m+1) \times (4m+2)$ binary matrix 
whose columns represent the vectors from $S$. 
We need to prove that $M$ has a set of $2m$ columns 
whose sum is a zero vector. 

By Theorem~\ref{th:s_2m_2m}, $4m+2 \geq \s_{2m}(2m)$. 
If the rank of $M$ is less than $2m+1$, 
then by Lemma~\ref{th:rank}, there is a set of $2m$ columns 
that sum up to a zero vector. 
Suppose that the rank of $M$ is $2m+1$. 
Let $x\in\Z_2^{2m+1}$ be the sum of all columns of $M$, 
and $x_{4m+2}\in\Z_2^{2m+1}$ be the last column.
Add $x + x_{4m+2}$ to each column of $M$, 
and expand the matrix by a new column equal to $x$. 
In the resulting $(2m+1) \times (4m+3)$ matrix $M'$, 
the sum of all columns is a zero vector, 
and the last two columns are equal to $x$. 
$M'$ satisfies conditions of Lemma~\ref{th:Enomoto2} 
and can be brought to binormal form $M''$ by 
permutations of the columns and additions of one row to another. 
There are two identical columns in $M''$ 
that originated from the last two columns of $M'$. 
By Lemma~\ref{th:2m_1_4m_3}, there is a set of $2m$ columns in $M''$ 
with zero sum 
which includes at most one of these two columns. 
By Observation~\ref{th:operations}, 
$M'$ has a set of $2m$ columns with zero sum 
that does not include the last column. 
Therefore, $M$ has a set of $2m$ columns with zero sum.
\end{proof}

\section*{Acknowledgment}

The author thanks the anonymous referees for careful reading 
and helpful suggestions.

\bibliographystyle{elsarticle-num-names-alphsort}
\bibliography{EGZ5}

\end{document}